\documentclass[preprint,12pt]{elsarticle}
\usepackage{graphicx}
\usepackage{latexsym}
\usepackage{amssymb}
\usepackage{amsmath}
\usepackage{amsthm}
\usepackage{verbatim}
\usepackage{appendix}
\usepackage[text={145mm,220mm},centering]{geometry}
\usepackage{color}
\usepackage[mathlines]{lineno}

\numberwithin{equation}{section}

\newtheorem{theorem}{Theorem}[section]
\newtheorem{lemma}{Lemma}[section]

\newtheorem{remark}{Remark}[section]
\newtheorem{definition}{Definition}[section]

\newcommand{\R}{\mathbb{R}}

\journal{Elsevier}

\begin{document}

\begin{frontmatter}

\title{Breather solutions for a radially symmetric curl-curl wave equation with double power nonlinearity
}

\author[ad1]{Xin Meng}
\ead{mengxin22@mails.jlu.edu.cn}
\author[ad2]{Shuguan Ji\corref{cor}}
\ead{jisg100@nenu.edu.cn}
\address[ad1]{School of Mathematics, Jilin University, Changchun 130012, China}
\address[ad2]{School of Mathematics and Statistics and Center for Mathematics and Interdisciplinary Sciences, Northeast Normal University, Changchun 130024, China}
\cortext[cor]{Corresponding author.}

\begin{abstract}
This paper is concerned with breather solutions of a radially symmetric curl-curl wave equation with double power nonlinearity
\begin{equation*}
  \rho(x)\mathbf{u}_{tt}+\nabla\times(\mathbf{M}(x)\nabla\times \mathbf{u})+\mu (x)\mathbf{u}+v_p(x)|\mathbf{u}|^{p-1}\mathbf{u}+v_q(x)|\mathbf{u}|^{q-1}\mathbf{u}=0,
\end{equation*}
where $(x, t)\in\mathbb{R}^3\times\mathbb{R}$, $\mathbf{u}: \mathbb{R}^3\times\mathbb{R}\rightarrow \mathbb{R}^3$ is the unknown function, $\mathbf{M}: \mathbb{R}^3\rightarrow\R^{3\times3}$ and $\rho, \mu, v_p, v_q:\mathbb{R}^3\rightarrow (0, +\infty)$ are radially symmetric coefficient functions with $1<p<q$.
By considering the solutions with a special form $\mathbf{u}=y(|x|, t)\frac{x}{|x|}$, we obtain a family of ordinary differential equations (ODEs) parameterized by the radial variable $r=|x|$. Then we characterize periodic behaviors and analyze the joint effects of the double power nonlinear terms on the minimal period and the maximal amplitude. Under certain conditions, we construct a $2\pi\sqrt{\rho(0)/\mu(0)}$-periodic breather solution for the original curl-curl wave equation and find such a solution which can generate a continuum of phase-shifted breathers.
\end{abstract}


\end{frontmatter}

\section{Introduction}
Consider the radially symmetric curl-curl wave equation with double power nonlinearity
\begin{equation}\label{eq1.1}
  \rho(x)\mathbf{u}_{tt}+\nabla\times(\mathbf{M}(x)\nabla\times \mathbf{u})+\mu (x)\mathbf{u}+v_p(x)|\mathbf{u}|^{p-1}\mathbf{u}+v_q(x)|\mathbf{u}|^{q-1}\mathbf{u}=0
\end{equation}
for $(x,t)\in \R^3 \times \R$, where $\mathbf{u}: \mathbb{R}^3\times\mathbb{R}\rightarrow \mathbb{R}^3$ is the unknown function, $\mathbf{M}: \mathbb{R}^3\rightarrow\R^{3\times3}$, and $\rho, \mu, v_p, v_q:\R^3\rightarrow(0,\infty)$ are positive, radially symmetric functions with $1<p<q$. We consider the classical real-valued solutions $\mathbf{u}:\R^3\times\R\rightarrow\R^3$ which are $T$-periodic in time and spatially exponentially localized, i.e., sup$_{\R^3\times\R}|\mathbf{u}(x,t)|e^{\gamma|x|}<\infty$ for some $\gamma>0$. Such solutions are usually referred to as breathers or breather solutions.

Breather solutions are of great significance in physics, biology and nonlinear optics, and attract extensive attentions of both mathematicians and physicists (see \cite{20062,AL,CW,CNSNS,CNSN}). It was introduced by Ablowitz et al. \cite{1973} in the context of the (1+1)-dimensional sine-Gordon equation
\begin{equation*}
u_{tt}-u_{xx}+\sin u=0\quad\mbox{in}\ \R\times\R,
\end{equation*}
which possesses the breather families
\begin{equation*}
u(x,t;\omega)=4\arctan\left(\frac{\sqrt{1-\omega^2}\cos(\omega t)}{\omega\cosh(\sqrt{1-\omega^2}x)}\right)
\end{equation*}
with $0<\omega<1$.

It is well known that breather solutions exist in various systems. For example, nonlinear wave equations and Schr\"odinger equations on discrete lattices can support breather solutions. MacKay and Aubry \cite{19942} found that breather solutions exist in a broad range of time-reversible or Hamiltonian networks with small coupling constant. In \cite{2009}, the existence of breathers is shown in some cases for any value of the coupling constants, which generalizes the existence results obtained in \cite{19942}. In Schr\"odinger equations (see \cite{2010, NSci, JNSci, 2005, 2006, shihp}), the standing wave ansatz $u(x,t)=u(x)\exp(i\omega t)$ (see \cite{S}) is usually a common tool to prove the existence of breathers. Comparing with the discrete cases, breathers in continuous cases are relatively little known, so here we make an effort to study breathers of nonlinear wave equations in continuous situations while many methods developed in discrete lattice systems are not applicable.

Our interest in breathers of nonlinear wave equations originate from the fact that they cannot be observed in linear dispersion equations and are therefore a truly nonlinear phenomenon. The papers \cite{1994,1993} deal with perturbations of the sine-Gordon breathers and show that the family of breathers are nonpersistent under any nontrivial perturbation.
Recently, Blank et al. \cite{2011} used spatial dynamics, center manifold theory and bifurcation theory to construct such time periodic solutions for a special inhomogeneous nonlinear wave equations.
Inspiringly, Hirsch et al. \cite{2019} consider a more general situation with $x$-dependent coefficients and power-type nonlinearities like
\begin{equation}\label{nw}
\rho(x)u_{tt}-u_{xx}+\mu(x)u=\pm \Gamma(x)|u|^{p-1}u, \quad\mbox{in}\ \R\times\R,
\end{equation}
where $\Gamma(x)$ is a $2\pi$-periodic continuous positive function and $1<p<p^*$ for some $p^*$ depends on the choice of $\rho(x)$ and $\mu(x)$. However, only the special case of $p=3$ is considered in \cite{2011}. Actually, we note that \eqref{nw} is a particular case of \eqref{eq1.1} with $\mathbf{M}$ represents the identity matrix and $v_q(x)=0$, and for solutions with the form
\begin{equation*}
\mathbf{u}(x,t)=\begin{pmatrix}0\\ 0\\ u(x_1,t)\end{pmatrix} .
\end{equation*}
A field of this form is divergence-free, i.e., div $\mathbf{u}$=0, and hence
\begin{equation*}
\nabla\times\nabla\times\mathbf{u}(x,t)=\begin{pmatrix}0\\ 0\\ -\partial_{x_1}^2u(x_1,t)\end{pmatrix}.
\end{equation*}
It is assumed in \cite{2011} and \cite{2019} that the potential $\mu(x)$ is proportional to $\rho(x)$ which reduces the Fourier-transformed wave operator to a family of Hill-type ODE operators, thus the spectral analysis is significantly simplified. The variational method was applied in \cite{2019} so that the nonlinear term has no other smoothness assumptions except the continuity and superlinearity at zero and infinity. However, the above cases are only true in one dimensional space and it is difficult to be generalized to higher spatial dimensions.

In fact, for higher dimensional nonlinear wave equations, there are few results on the existence of breather solutions. Recently, Scheider \cite{2020} has constructed the weakly localized breathers by Fourier-expansion in time and bifurcation techniques for three dimensional case. For the spatial dimension $N\geq2$, Mandel and Scheider \cite{2021} obtained breather solutions that are also weakly localized in a distributional sense and polychromatic by dual variational method.
It is entirely unclear whether strongly localized breathers of nonlinear wave equations exist in these cases. However, such strongly localized breather solutions play an important role in theoretical scenarios where photonic crystals are used as optical storage \cite{2007}. As far as we know, the existing results dealing with strongly localized breathers of nonlinear wave equations in higher spatial dimensions can be seen in \cite{2016}. Inspired by the previous results, we intend to study a class of more general (3+1)-dimensional radially symmetric wave equations with double power nonlinearity involving curl-curl operators and prove the existence of strongly localized breathers, i.e., $u(\cdot,t)\in C^2(\R^3)$ for almost all $t\in\R$ and sup$_{\R^3\times\R}|\mathbf{u}(x,t)|e^{\gamma|x|}<\infty$ for some $\gamma>0$.

The particular feature of the curl-curl operator $\nabla\times(\mathbf{M}(x)\nabla\times)$ of \eqref{eq1.1} arises in specific models in the Maxwell equations as a physical motivation (see \cite{ABD,20162,M,PS}). Our interest in breathers under the context of the curl-curl nonlinear wave equations originates from optical breathers of the Maxwell equations in anisotropic materials, where the permittivity nonlinearly depends on the electromagnetic fields (see \cite{20062}). However, the curl-curl operator exhibits major mathematical challenges. The main difficulty is that the $\nabla\times(\mathbf{M}(x)\nabla\times)$ operator has an infinite-dimensional kernel which causes the energy functional associated with \eqref{eq1.1} is strongly indefinite. To overcome these difficulties, Azzollini et al. \cite{ABD} studied the curl-curl problems in $\R^3$ in the cylindrically symmetric setting, and then the method has been extensively used in the curl-curl problems related to the nonlinear Maxwell equations (see \cite{20162,MMM}).

In this paper, we use the radially symmetric setting which is motivated by \cite{20162}. In this case, the gradient field is annihilated by the curl-curl operator due to the radially symmetric assumptions of coefficients. Then the (3+1)-dimensional wave equation reduces to a family of ODEs. We characterize periodic behaviors and analyze the joint effects of the double power nonlinear terms on the minimal period and the maximal amplitude, and establish the existence and exponential decay properties of real-valued breather solutions of equation \eqref{eq1.1}.  We also prove the existence of breathers of the type $e^{i\frac{2\pi}{T}t}\mathbf{u}(x)$ under various assumptions on the coefficients for the vector-valued wave equation \eqref{eq1.1}.

This paper is organized as follows. In Section 2, we shall give some basic definitions and the construction of breather solutions, then obtain a family of ODEs possessing a first integral. In Section 3, we use the qualitative analysis method of ODEs to characterize periodic behaviors, and analyze the joint effects of the double power nonlinear terms on the minimal period and the maximal amplitude. In Section 4, we prove our main theorems for the existence and exponential decay properties of breather solutions of \eqref{eq1.1}.  Section 5 is a brief conclusion.

\section{Definitions and preliminaries}

In this section, we will give some useful definitions and the construction for breather solutions of \eqref{eq1.1}. By the setting of the construction which depends on the spatial variable, the gradient fields are annihilated by the curl-curl operator. We eventually obtain a family of ODEs which can be reduced to a unified second-order differential equation possessing a first integral through some transformations. This greatly simplifies the original curl-curl wave equations.

\begin{definition}\label{de2.2}
For a radially symmetric $C^2$-function $U:\R^3\rightarrow \R$, let $\tilde{U}:[0,\infty)\rightarrow\R$ with $\tilde{U}(|x|)=U(x)$ be its one-dimensional representative, then we have $\tilde{U}\in C^2([0,\infty))$ and $\tilde{U}'(0)=0$.
\end{definition}

\begin{lemma}\label{lemma2.1}
Let $\psi:[0,\infty)\rightarrow\R$ be a $C^2$-function and let $F:\R^3\setminus\left\{0\right\}\rightarrow \R^3$ be given by $F(x):=\psi(|x|)\frac{x}{|x|}$. Then $F$ can be extended to a function
\\
(i) $F\in C^1(\R^3)$ if and only if $\psi(0)=0$;
\\
(ii) $F\in C^2(\R^3)$ if and only if $\psi(0)=\psi''(0)=0$.
\end{lemma}

The proof of this lemma is nearly identical to that of \cite[lemma 1]{2016} and hence is omitted here. In what follows, we use the function $y=y(r,t)$ from $[0,\infty)\times\R$ to $\R$ with the notation $'$ to denote the differential form acting on the spatial variable and $\dot{}$ to denote the differential form acting on the time variable.

Under the radially symmetric assumptions of the coefficients in \eqref{eq1.1}, we construct classical solutions $\mathbf{u}$ to \eqref{eq1.1} with the form
\begin{equation}\label{uconstruct}
\mathbf{u}(x,t):=y(|x|,t)\frac{x}{|x|}.
\end{equation}
We claim that $\mathbf{u}(x,t)$ with the form \eqref{uconstruct} is a $C^2(\R^3\times\R)$ function of the variables $x$ and $t$ if $y:[0,\infty)\times\R\rightarrow\R$ is a $C^2$-function satisfying
\begin{equation*}
y(0,t)=y''(0,t)=0.
\end{equation*}
It is apparently by Lemma \ref{lemma2.1}. We also have $\dot{y}(0,t)=\ddot{y}(0,t)=0$. Moreover, by the construction of $\mathbf{u}$, we obtain that it is a gradient-field, i.e., $\mathbf{u}(x,t)=\nabla_xY(|x|,t)$ where $Y(r,t)=\int_{0}^{r}y(s,t)ds$. In this case we note that $\mathbf{M}(x)\nabla\times \mathbf{u}=0$ where $\mathbf{M}(x)\in\R^{3\times3}$ is an arbitrary $3\times3$ matrix function. Substituting the ansatz of $\mathbf{u}(x,t)$ of the form \eqref{uconstruct} into equation \eqref{eq1.1}, we have
\begin{equation}\label{eq1.2}
\tilde{\rho}(r)\ddot{y}+\tilde{\mu}(r)y+\tilde{v}_p(r)|y|^{p-1}y+\tilde{v}_q(r)|y|^{q-1}y=0, \quad \mbox{for}\  r\geq0, \ t\in\R.
\end{equation}
Here we note that \eqref{eq1.2} is autonomous with respect to $t$ and that $r=|x|\geq0$ is a parameter. It is easy to see that $\mathbf{u}(x,t)$ solves equation \eqref{eq1.1} if and only if $y$ satisfies equation \eqref{eq1.2}.

We suppose that the coefficients $v_p(x)$ and $v_q(x)$ of the double power nonlinearities in \eqref{eq1.1} satisfy

$(V)$ $v_p(x)=\left( \frac{\mu(x)^{q-p}}{v_q(x)^{1-p}}\right)^{\frac{1}{q-1}}$ for $x\in \R^3$ and $1<p<q$.

For convenience, we use the following abbreviation
\begin{equation*}
\tau(x)=\left (\frac{ \mu(x)}{v_q(x)}\right )^{\frac{1}{q-1}}\quad\rm{and}\quad \sigma(x)=\left(\frac{ \mu(x)}{\rho(x)}\right)^{\frac{1}{2}} .
\end{equation*}
Since $\rho, \mu, v_p, v_q$ are positive, radially symmetric functions, we denote
\begin{equation}\label{eq1.5}
\tilde{v}_p(r)=\left( \frac{\tilde{\mu}(r)^{q-p}}{\tilde{v}_q(r)^{1-p}}\right)^{\frac{1}{q-1}}, \quad\tilde{\tau}(r)=\left (\frac{ \tilde{\mu}(r)}{\tilde{v}_q(r)}\right )^{\frac{1}{q-1}},\quad \tilde{\sigma}(r)=\left(\frac{ \tilde{\mu}(r)}{\tilde{\rho}(r)}\right)^{\frac{1}{2}}.
\end{equation}

Since we intend to apply the ODEs technique to study breathers which relies on rescaling \eqref{eq1.2} to a relatively simple form, we need to find solutions $y(r,t)$ of \eqref{eq1.2} with the form
\begin{equation}\label{eq1.3}
  y(r,t)=\tilde{\tau}(r)\phi(\tilde{\sigma}(r)t).
\end{equation}
Plugging \eqref{eq1.3} into \eqref{eq1.2} and comparing the spatially varying coefficients, we get
\begin{equation}\label{eq1.4}
\ddot{\phi}+\phi+|\phi|^{p-1}\phi+|\phi|^{q-1}\phi=0.
\end{equation}
By multiplying the left and right ends of \eqref{eq1.4} by $2\dot{\phi}$, it is easy to obtain that
\begin{equation*}
\dot{\phi}^2+\phi^2+\frac{2}{p+1}|\phi|^{p+1}+\frac{2}{q+1}|\phi|^{q+1}=c
\end{equation*}
for some constant $c\in[0,\infty)$.
\begin{definition}\label{de2.3}
Define the function $B:\R^2\rightarrow\R$ by
\begin{equation}\label{eq1.6}
B(\alpha,\beta):=\beta^2+\alpha^2+\frac{2}{p+1}|\alpha|^{p+1}+\frac{2}{q+1}|\alpha|^{q+1}.
\end{equation}
Then $B$ is a first integral for \eqref{eq1.4}, i.e., every solution $\phi$ of \eqref{eq1.4} satisfies $B(\phi,\dot{\phi})=c$.
\end{definition}

\section{Qualitative analysis}
In what follows, with the help of the first integral \eqref{eq1.6} and by using the qualitative analysis method of ODEs, we characterize periodic behaviors of solutions of equation \eqref{eq1.4}, and discuss the joint effects of the double nonlinear terms on the minimal period $P(c)$ and the maximal amplitude $A(c):=\max_{t\in\R}|\phi(t)|$ for every periodic solution $\phi(t)$ of \eqref{eq1.4} characterized by a fixed $c$.

\begin{lemma}\label{lemma3.1}
$A(c)$ enjoys the following properties:

$(1)$ $A\in C([0,\infty))\bigcap C^{\infty}(0,\infty)$ and $A([0,\infty))=[0,\infty)$;

$(2)$ $A'(c)>0$ and $A(c)\leq\sqrt{c}$ for all $c>0$; and

$(3)$ $\lim_{c\rightarrow0}\frac{A(c)}{\sqrt{c}}=1$ and $\lim_{c\rightarrow\infty}A(c)=\infty$.

\end{lemma}
\begin{proof}
The maximal amplitude $A(c)$ is given by
\begin{equation}
\label{dA}
B(A(c),0)=A(c)^2+\frac{2}{p+1}A(c)^{p+1}+\frac{2}{q+1}A(c)^{q+1}=c,
\end{equation}
which implies the strict monotonicity, differentiability and continuity of $A(c)$ with respect to $c>0$. It also provides the inequality $A(c)\leq\sqrt{c}$ since $A(c)^2\leq c$ and $A(c)\geq0$. Moreover, we have $\lim_{c\rightarrow0}\frac{A(c)}{\sqrt{c}}=1$ and $\lim_{c\rightarrow\infty}A(c)=\infty$ by \eqref{dA}. This completes the proof.
\end{proof}

\begin{lemma}\label{lemma3.2}

$P(c)$ enjoys the following properties:

$(1)$ $P\in C([0,\infty))\bigcap C^{\infty}((0,\infty))$ and $P([0,\infty))=(0,2\pi]$; and

$(2)$ $\lim_{c\rightarrow\infty}P(c)=0$ and $P(0)=2\pi$.

\end{lemma}

\begin{proof}
We verify the properties for $P(c)$ by using the first integral
\begin{equation*}
|\dot{\phi}|^2+|\phi|^2+\frac{2}{p+1}|\phi|^{p+1}+\frac{2}{q+1}|\phi|^{q+1}=c
\end{equation*}
\begin{figure}[htbp]
  \centering
  \includegraphics[width=5cm]{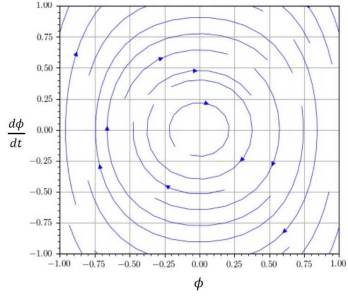}
  \caption{Part of the phase plane of \eqref{eq1.4} for p=3 and q=4.}\label{f1}
\end{figure}
to solve for $\dot{\phi}$ in all four quadrants of the phase-plane, cf. Figure \ref{f1}. The function $P(c)$ is given implicitly through the defining equation (\ref{dA}) for $A(c)$, which yields
\begin{eqnarray}
\nonumber P(c)&=&4\int_{0}^{A(c)}\frac{1}{\sqrt{c-\phi^2-\frac{2}{p+1}\phi^{p+1}-\frac{2}{q+1}\phi^{q+1}}}d\phi
\\\nonumber&=&4\int_{0}^{1}\frac{A(c)}{\sqrt{c-A(c)^2z^2-\frac{2}{p+1}A(c)^{p+1}z^{p+1}-\frac{2}{q+1}A(c)^{q+1}z^ {q+1}}}dz \label{eq2.5}
\\&=&4\int_{0}^{1}\frac{1}{\sqrt{1-z^2+\frac{2}{p+1}A(c)^{p-1}(1-z^{p+1})+\frac{2}{q+1}A(c)^{q-1}(1-z^{q+1})}}dz.~~~~~~
\end{eqnarray}
It shows that $P(c)$ has the asserted smoothness. By the properties of $A(c)$, we also find that $\lim_{c\rightarrow\infty}P(c)=0$ and
\begin{equation*}
\lim_{c\rightarrow0}P(c)=4\int_{0}^{1}\frac{1}{\sqrt{1-z^2}}dz=2\pi.
\end{equation*}
The proof is completed.
\end{proof}

\begin{lemma}\label{lemma2.5} $P(c)$ is invertible and its inverse  $Q=P^{-1}:(0,2\pi]\rightarrow[0,\infty)$ is $C^{\infty}$ on $(0,2\pi)$ and has the following expansions as $s\rightarrow2\pi^-$,
\begin{eqnarray*}
 Q(s)&=&\lambda(2\pi-s)^{\frac{2}{q-1}}(1+O(2\pi-s)),
\\\sqrt{Q(s)}&=&\sqrt{\lambda}(2\pi-s)^{\frac{1}{q-1}}(1+O(2\pi-s)),
\\Q'(s)&=&-\frac{2\lambda}{q-1}(2\pi-s)^{\frac{3-q}{q-1}}(1+O(2\pi-s)),
\\\sqrt{Q(s)} '&=&-\frac{\sqrt{\lambda}}{q-1}(2\pi-s)^{\frac{2-q}{q-1}}(1+O(2\pi-s)),
\\Q''(s)&=&\frac{2\lambda(3-q)}{(q-1)^2}(2\pi-s)^{\frac{4-2q}{q-1}}(1+O(2\pi-s)),
\\\sqrt{Q(s)}''&=&\frac{\sqrt{\lambda}(2-q)}{(q-1)^2}(2\pi-s)^{\frac{3-2q}{q-1}}(1+O(2\pi-s))
\end{eqnarray*}
for some constant $\lambda>0$.

\end{lemma}

\begin{proof}
In view of \eqref{eq2.5}, we define
\begin{equation*}
P(c)=W\left(\frac{2}{q+1}A(c)^{q-1}\right),
\end{equation*}
where
\begin{eqnarray*}
W(u)&=&4\int_{0}^{1}\frac{1}{\sqrt{1-z^2+\frac{2}{p+1}(1-z^{p+1})(\frac{q+1}{2})^{\frac{p-1}{q-1}}u^{\frac{p-1}{q-1}}+u(1-z^{q+1})}}dz
\\&=&4\int_{0}^{1}\frac{1}{\sqrt{1-z^2}\sqrt{1+m(z)\left(u+\frac{2}{p+1}(\frac{q+1}{2})^{\frac{p-1}{q-1}}u^{\frac{p-1}{q-1}}n(z)\right)}}dz
\end{eqnarray*}
with $u=\frac{2}{q+1}A(c)^{q-1}\geq0$, and $m(z)=\frac{1-z^{q+1}}{1-z^2}$ and $n(z)=\frac{1-z^{p+1}}{1-z^{q+1}}$ are continuous and positive functions on $[0,1]$ and $[1,\infty)$ respectively. Thus we have $W(0)=2\pi$, $W(\infty)=0$ and $W\in C^{\infty}[0,\infty)$. Moreover, $W$ is strictly decreasing and convex since
\begin{eqnarray}
\nonumber W'(u)&=\ &-2\int_{0}^{1}\frac{m(z)\left ( 1+\frac{2(p-1)}{(p+1)(q-1)}(\frac{q+1}{2})^{\frac{p-1}{q-1}}u^{\frac{p-q}{q-1}}n(z)\right )}{\sqrt{1-z^2}\left(1+m(z)\left(u+\frac{2}{p+1}(\frac{q+1}{2})^{\frac{p-1}{q-1}}u^{\frac{p-1}{q-1}}n(z)\right)\right) ^{\frac{3}{2}}}dz
\\&=&-2\int_{0}^{1}\frac{h(u,z)}{\sqrt{1-z^2}g(u,z)^{\frac{3}{2}}}dz<0 \label{eq2.6}
\end{eqnarray}
for $1<p<q$, where $h(u,z)=m(z)\left (1+\frac{2(p-1)} {(p+1)(q-1)}(\frac{q+1}{2})^{\frac{p-1}{q-1}}u^{\frac{p-q}{q-1}}n(z)\right )$ and $g(u,z)=1+m(z)\left(u+\frac{2}{p+1}(\frac{q+1}{2})^{\frac{p-1}{q-1}}u^{\frac{p-1}{q-1}}n(z)\right)$, and
\begin{eqnarray}
W''(u)&=&-2\int_{0}^{1}\frac{h'(u,z)g(u,z)^{\frac{3}{2}}-h(u,z)\frac{3}{2}g(u,z)^{\frac{1}{2}}h(u,z)}{\sqrt{1-z^2}g(u,z)^{3}}dz
\\\nonumber &=&-2\int_{0}^{1}\frac{m(z)\frac{2(p-1)(p-q)}{(p+1)(q-1)^2}(\frac{q+1}{2})^{\frac{p-1}{q-1}}u^{\frac{p-2q+1}{q-1}}g(u,z)- \frac{3}{2}h(u,z)^2}{\sqrt{1-z^2}g(u,z)^{\frac{5}{2 }}}dz>0.
\end{eqnarray}
Hence we obtain that $W^{-1}\in C^{\infty}((0,2\pi])$. Furthermore, in view of \eqref{eq2.6} together with the fact that $A'(c)>0$ on $(0,\infty)$, it is easy to see that $P'(c)<0$ on $(0,\infty)$ which implies the invertibility of $P$.

We then do some preparations for studying $P^{-1}$. From equation (\ref{dA}) and $u=\frac{2}{q+1}A(c)^{q-1}$, we can easily verify
\begin{eqnarray*}
c&=\ &A(c)^2+\frac{2}{p+1}A(c)^{p+1}+\frac{2}{q+1}A(c)^{q+1}
\\&=\ &(\frac{q+1}{2})^{\frac{2}{q-1}}(u^{\frac{2}{q-1}}+u^{\frac{q+1}{q-1}})+ \frac{2}{p+1}(\frac{q+1}{2})^{\frac{p+1}{q-1}}u^{\frac{p+1}{q-1}}
\\&=:&\Psi(u).
\end{eqnarray*}
Then we denote
\begin{equation*}
P(c)=W(u)=W(\Psi^{-1}(c))\quad\mbox{and}\quad Q=P^{-1}=\Psi\circ W^{-1}.
\end{equation*}
From the Taylor approximation for $W^{-1}$ at $2\pi$ with $\tilde{\lambda}=-\frac{1}{W'(0)}>0$ and $\tilde{\theta}=W''(0)>0$, we find that as $s\rightarrow2\pi^{-}$,
\begin{eqnarray*}
W^{-1}(s)&=&W^{-1}(2\pi)+(W^{-1})'(2\pi)(s-2\pi)+O((2\pi-s)^2)
\\&=&\tilde{\lambda }(2\pi-s)(1+O(2\pi-s)),
\\(W^{-1})'(s)&=&\frac{1}{W'(W^{-1}(s))}=-\tilde{\lambda}(1+O(2\pi-s)),
\\(W^{-1})''(s)&=&-\frac{W''(W^{-1}(s))}{\left(W'(W^{-1}(s))\right)^3}=\tilde{\theta}\tilde{\lambda}^3(1+O(2\pi-s)).
\end{eqnarray*}
Thus, as $s\rightarrow2\pi^{-}$ we have
\begin{eqnarray*}
Q(s)&=&\Psi(W^{-1}(s))
\\&=&(\frac{q+1}{2})^{\frac{2}{q-1}}\left((W^{-1}(s))^{\frac{2}{q-1}}+(W^{-1}(s))^{\frac{q+1}{q-1}}\right)\\
&&+\frac{2}{p+1}(\frac{q+1}{2})^{\frac{p+1}{q-1}}(W^{-1}(s))^{\frac{p+1}{q-1}}
\\&=&\lambda (2\pi-s)^{\frac{2}{q-1}}(1+O(2\pi-s))
\end{eqnarray*}
with $\lambda=(\frac{q+1}{2})^{\frac{2}{q-1}}\tilde{\lambda}^{\frac{2}{q-1}}>0$. Note that we also have
\begin{eqnarray*}
Q'(s)&=&\Psi'(W^{-1}(s))(W^{-1})'(s))
\\&=&(\frac{q+1}{2})^{\frac{2}{q-1}}\left(\frac{2}{q-1}(W^{-1}(s))^{\frac{3-q}{q-1}}+\frac{q+1}{q-1}(W^{-1}(s))^{\frac{2}{q-1}}\right)\\
&&\times(-\tilde{\lambda})(1+O(2\pi-s))
\\&&+\frac{2}{q-1}(\frac{q+1}{2})^{\frac{p+1}{q-1}}(W^{-1}(s))^{\frac{p-q+2}{q-1}}(-\tilde{\lambda})(1+O(2\pi-s))
\\&=&-\frac{2\lambda}{q-1}(2\pi-s)^{\frac{3-q}{q-1}}(1+O(2\pi-s))
\end{eqnarray*}
and
\begin{eqnarray*}
Q''(s)&=&\Psi''(W^{-1}(s))\left((W^{-1})'(s)\right)^2+\Psi'(W^{-1}(s))(W^{-1})''(s)
\\&=&(\frac{q+1}{2})^{\frac{2}{q-1}}\left(\frac{2(3-q)}{(q-1)^2}(W^{-1}(s))^{\frac{4-2q}{q-1}}+\frac{2(q+1)}{(q-1)^2}(W^{-1}(s))^ {\frac{3-q}{q-1}}\right)\\
&&\times\tilde{\lambda}^2(1 +O(2\pi-s))+O((2\pi-s)^{\frac{3-q}{q-1}})
\\&&+\frac{2(p-q+2)}{(q-1)^2}(\frac{q+1}{2})^{\frac{p+1}{q-1}}(W^{-1}(s))^{\frac{p-2q+3}{q-1}}\tilde{\lambda}^2(1+O(2\pi-s))\\
\\&=&\frac{2\lambda(3-q)}{(q-1)^2}(2\pi-s)^{\frac{4-2q}{q-1}}(1+O(2\pi-s))
\end{eqnarray*}
as $s\rightarrow2\pi^{-}$. Thus we can easily obtain the representations for $\sqrt{Q}$, $\sqrt{Q}'=\frac{Q'}{2\sqrt{Q}}$ and $\sqrt{Q} ''=\frac{1}{2Q^{3/2}}(Q''Q-\frac{1}{2}(Q')^2)$ by substituting the above representations of $Q$, $Q'$ and $Q''$. The proof is completed.
\end{proof}

\section{Main results}

In this section, we shall present our main results for the existence of breathers of \eqref{eq1.1} and the corresponding proofs. Recall that, for a $C^2$-function $U:\R^3\rightarrow\R$, $U(x)\rightarrow0$ in the $C^2$-sense as $x\rightarrow0$ means $U(x)\rightarrow0,\ \nabla U(x)\rightarrow0$ and $\ D^2U(x)\rightarrow0$ as $x\rightarrow0.$

\begin{theorem}\label{th4.1}
Assume $\rho(x), \mu(x), v_p(x), v_q(x):\R^3\rightarrow(0,\infty)$ are radially symmetric $C^2$-functions satisfying $(V)$ and let $T=2\pi\sigma(0)$. If the following four conditions hold:

$(H1)$ $\sigma(0)\sigma(x)<1$ for all $x\in\R^3\setminus \{0\}$,

$(H2)$ $|1-\sigma(0)\sigma(x)|^{\frac{1}{q-1}}\rightarrow0$ in the $C^2$-sense as $x\rightarrow0$,

$(H3)$ $\sup_{x\in\R^3}|1-\sigma(0)\sigma(x)|e^{\gamma(q-1)|x|}<\infty$ for some $\gamma>0$, and

$(H4)$ $\sup_{x\in\R^3}\tau(x)<\infty$,
\\then there exists a $T$-periodic $\R^3$-valued breather solution $\mathbf{u}$ of \eqref{eq1.1} with the extra property that $\sup_{\R^3\times\R}|\mathbf{u}(x,t)|e^{\gamma|x|}<\infty$. The breather solution $\mathbf{u}$ can generate a continuum of phase-shifted breathers $\mathbf{u}_b(x,t)=\mathbf{u}(x,t+b(x))$ where $b:\R^3\rightarrow\R$ is an arbitrary radially symmetric $C^2$-function.

\end{theorem}

\begin{proof} \textbf{Step 1}: Choose $c$ as a function of the radial variable $r\in[0,\infty)$.
We first select a $C^2$-curve $\xi:[0,\infty)\rightarrow\R^2$ in the phase space so that $B(\xi(c))=c^2$, where $B$ is the first integral as defined in Definition \ref{de2.3}. Then let us give an example to make it easy be understood such a curve, e.g., $\xi(c)=(0,c)$, which only selects a special case of the continuum of phase-shifted breathers as described. However there exist some other choices of $\xi$ and will be given later.

We denote the solution of \eqref{eq1.4} by $\phi(t;c)$ with $(\phi(0;c),\dot{\phi}(0;c))=\xi(c)$. Then $\phi:\R\times[0,\infty)\rightarrow\R$ is a $C^2$-function and $\phi(t;c)$ is $P(c^2)$-periodic in the $t$-variable.
Now let us define the solution $y$ of \eqref{eq1.2} by
\begin{equation}\label{eq3.1}
y(r,t)=\tilde{\tau}(r)\phi(\tilde{\sigma}(r)t;c)
\end{equation}
with $\tilde{\tau}(r)$ and $\tilde{\sigma}(r)$ as described in \eqref{eq1.5}. The necessary condition of $T$-periodicity of $y$ in the $t$-variable tells us how to choose $c$ as a function of the radial variable $r\in[0,\infty)$, i.e.,
\begin{equation}
d(r):=\tilde{\sigma}(r)T=P(c^2).
\end{equation}
By the properties of $Q=P^{-1}$ that, $Q: (0,2\pi]\rightarrow[0,\infty)$ is strictly decreasing and $C^{\infty}$ on $(0,2\pi)$ (see Lemmas \ref{lemma3.2} and \ref{lemma2.5}), we get
\begin{equation}\label{eq3.4}
c(r)=\left(Q(d(r))\right)^{\frac{1}{2}},
\end{equation}
which has to be inserted into \eqref{eq3.1}. We can easily verify that $c(r)$ is well-defined and $C^2$ on $(0,\infty)$ by the assumption $(H1)$.

\textbf{Step 2}: Existence of the breather solution with the assumed form $\mathbf{u}(x,t):=y(|x|,t)\frac{x}{|x|}$. Here we need to prove that $y:[0,\infty)\times\R\rightarrow\R$ is a $C^2$-function with $y(0,t)=y''(0,t)=0$ and even exponentially decays to 0 as $r\rightarrow0$. In view of Lemma \ref{lemma3.1} and the assumption $(H4)$, we have
\begin{eqnarray*}
|y(r,t)|\leq\tilde{\tau}(r)A(c(r)^2)\leq\tilde{\tau}(r)c(r)
\leq G\left(Q(d(r))\right)^{\frac{1}{2}},
\end{eqnarray*}
where $G$ is a constant. The assumptions $(H2)$ and $(H3)$ yield that $d(r)$ tends to $2\pi$ as $r\rightarrow\infty$ and as $r\rightarrow0$, then we can easily verify the estimate
\begin{equation}\label{eq3.5}
|y(r,t)|\leq G\sqrt{\lambda}(2\pi-d(r))^{\frac{1}{q-1}}O(1)\quad\mbox{as}\ r\rightarrow\infty\ \mbox{and}\ \mbox{as}\ r\rightarrow0.
\end{equation}
By the assumption $(H3)$, we obtain that $|y(r,t)|\leq Ce^{-\gamma r}$ for $r\geq0$, with some positive constant $C$. This implies the exponential decay of $\mathbf{u}(x,t)=y(|x|,t)\frac{x}{|x|}$ as $|x|\rightarrow\infty$.

In view of \eqref{eq3.5} and the assumption $(H2)$, we can easily see that $y(0,t)=0$, then we only need to prove $y:[0,\infty)\times\R\rightarrow\R$ is a $C^2$-function and $y''(0,t)=0$. For this, by \eqref{eq3.4} we deduce that
\begin{equation*}
c(r)=\sqrt{\lambda}(2\pi-d(r))^{\frac{1}{q-1}}O(1)\rightarrow0\quad\mbox{as}\ r\rightarrow0.
\end{equation*}
Furthermore, by Lemma \ref{lemma2.5} and the assumption $(H2)$, we have
\begin{eqnarray*}
c'(r)&=&\sqrt{Q}'(d(r))d'(r)
\\&=&-\frac{\sqrt{\lambda}}{q-1}(2\pi-d(r))^{\frac{2-q}{q-1}}O(1)d'(r)
\\&=&\sqrt{\lambda}\left((2\pi-d(r))^{\frac{1}{q-1}}\right)'O(1)
\\&=&o(1)\ \ as\ r\rightarrow0.
\end{eqnarray*}
 Moreover, we obtain
\begin{eqnarray*}
c''(r)&=&\sqrt{Q}''(d(r))d'(r)^2+\sqrt{Q}'(d(r))d''(r)
\\&=&\frac{\sqrt{\lambda}(2-q)}{(q-1)^2}(2\pi-d(r))^{\frac{3-2q}{q-1}}(1+O(2\pi-d(r)))d'(r)^2
\\&&-\frac{\sqrt{\lambda}}{q-1}(2\pi-d(r))^{\frac{2-q}{q-1}}(1+O(2\pi-d(r)))d''(r)
\\&=&\sqrt{\lambda}\left((2\pi-d(r))^{\frac{1}{q-1}}\right)''+O(1)\frac{\sqrt{\lambda}(2-q)}{(q-1)^2}(2\pi-d(r))^{\frac{2-q}{q-1}}d'(r)^2
\\&&-O(1)\frac{\sqrt{\lambda}}{q-1}(2\pi-d(r))^{\frac{1}{q-1}}d''(r).
\end{eqnarray*}

In view of the assumption $(H2)$, the term $L_1(r)=\left((2\pi-d(r))^{\frac{1}{q-1}}\right)''\rightarrow0$ as $r\rightarrow0$. By the definition of $d(r)$, we see that $d$ is a $C^2$-function on $[0,\infty)$. Thus the term $L_2(r)=(2\pi-d(r))^{\frac{2-q}{q-1}}d'(r)^2=(1-q)\left((2\pi-d(r))^{\frac{1}{q-1}}\right)'d'(r)\rightarrow0$ as $r\rightarrow0$, due to the fact that $d'$ is bounded near 0 together with the assumption $(H2)$. As for the term $L_3=(2\pi-d(r))^{\frac{1}{q-1}}d''(r)$, it converges to 0 since $d(r)\rightarrow2\pi$ as $r\rightarrow0$ and $d''$ is bounded near 0. Then the above equation implies $c''(r)\rightarrow0$ as $r\rightarrow0$. So that $c$ can be extended to a function $c\in C^2([0,\infty))$ with $c(0)=c'(0)=c''(0)=0$. From this and \eqref{eq3.1}, we derive $y(r,t)\in C^2([0,\infty)\times\R)$ with the presentation $y(r,t)=\tilde{\tau}(r)\phi(\tilde{\sigma}(r)t,c(r))$. Then we can deduce that
\begin{equation*}
y'(r,t)=\tilde{\tau}'(r)\phi(\tilde{\sigma}(r)t,c(r))+\tilde{\tau}(r)\dot{\phi}(\tilde{\sigma}(r)t,c(r)) \tilde{\sigma}'(r)t+\tilde{\tau}(r)\frac{\partial \phi}{\partial c}(\tilde{\sigma}(r)t,c(r))c'(r)
\end{equation*}
and
\begin{eqnarray*}
y''(0,t)&=&\tilde{\tau}''(0)\phi(\tilde{\sigma}(0)t,c(0))+2\tilde{\tau}'(0)\dot{\phi}(\tilde{\sigma}(0)t,c(0)) \tilde{\sigma}'(0)t
\\&&+2\tilde{\tau}'(0)\frac{\partial \phi}{\partial c}(\tilde{\sigma}(0)t,c(0))c'(0)+\tilde{\tau}(0)\ddot{\phi}(\tilde{\sigma}(0)t,c(0))\tilde{\sigma}'(0)^2t^2
\\&&+\tilde{\tau}(0)\dot{\phi}(\tilde{\sigma}(0)t,c(0))\tilde{\sigma}''(0)t+2\tilde{\tau}(0)\frac{\partial \dot{\phi}}{\partial c}(\tilde{\sigma}(0)t,c(0))\tilde{\sigma}'(0)tc'(0)
\\&&+\tilde{\tau}(0)\frac{\partial^2\phi}{\partial c^2}(\tilde{\sigma}(0)t,c(0))c'(0)^2+\tilde{\tau}(0)\frac{\partial \phi}{\partial c}(\tilde{\sigma}(0)t,c(0))c''(0)
\\&=&0,
\end{eqnarray*}
where we use $\phi(\cdot,0)=0$, $\dot{\phi}(\cdot,0)=0$, $\ddot{\phi}(\cdot,0)=0$ and $c(0)=c'(0)=c''(0)=0$. In view of Lemma \ref{lemma2.1}, we obtain that $\mathbf{u}(x,t)\in C^2(\R^3\times\R)$.

\textbf{Step 3}: We now show how to generate a continuum of different phase-shifted breathers from a radially symmetric breather solution.

Firstly, we discuss the choice of the initial curve $\xi(c)=(0,c)$ which gives rise to the solution family $\phi(t;c)$ such that $(\phi(0;c),\dot{\phi}(0;c))=\xi(c)$. This special selection is convenient but arbitrary. Then other possible choices of $\xi$ will be given, for example, by
\begin{equation*}
\tilde{\xi}(c):=(\phi(a(c);c),\dot{\phi}(a(c);c)),
\end{equation*}
where $a(c)$ is an arbitrary real $C^2$-function on $[0,\infty)$. We need to find some $C^2$-curve such that $B(\xi(c))=c^2$, where $B$ is the first integral of \eqref{eq1.4}. It is easy to see that $B(\tilde{\xi}(c))=B(\phi(a(c);c),\dot{\phi}(a(c);c))=c^2$. As for the new curve $\tilde{\xi}(c)$, we can give a new solution family $\tilde{\phi}(t;c)$ determined by the initial condition
\begin{equation*}
(\tilde{\phi}(0;c),\dot{\tilde{\phi}}(0;c))=\tilde{\xi}(c).
\end{equation*}
For the uniqueness of the initial value problem, the above two solution families have a simple relationship as
\begin{equation*}
\tilde{\phi}(t;c)=\phi(t+a(c);c).
\end{equation*}
By the choice of the new curve, we then compare the solutions $\mathbf{u}$ and $\tilde{\mathbf{u}}$ generated by $\xi$ and $\tilde{\xi}$, i.e.,
\begin{equation*}
\mathbf{u}(x,t)=\tilde{\tau}(r)\phi(\tilde{\sigma}(r)t;c(r))\frac{x}{r}
\end{equation*}
with $c(r)=\left(Q(d(r))\right)^{\frac{1}{2}}=\left(P^{-1}(\tilde{\sigma}(r)T)\right)^{\frac{1}{2}}$ and
\begin{eqnarray*}
\tilde{\mathbf{u}}(x,t)&=&\tilde{\tau}(r)\tilde{\phi}(\tilde{\sigma}(r)t;c(r))\frac{x}{r}
\\&=&\tilde{\tau}(r)\phi(\tilde{\sigma}(r)t+a(c(r));c(r))\frac{x}{r}
\\&=&y(|x|,t+\tilde{b}(r))\frac{x}{r}
\\&=&\mathbf{u}(x,t+\tilde{b}(r)),
\end{eqnarray*}
where $\tilde{b}(r)=a(c(r))/\tilde{\sigma}(r)$ is a $C^2$-function on $[0,\infty)$. Thus $b:\R^3\rightarrow\R$ is an arbitrary radially symmetric $C^2$-function generated by $b(x)=\tilde{b}(|x|)$ as described in Theorem \ref{th4.1}. Then the different choice of the initial curve can generate a continuum of phase-shifted breathers from one radially symmetric breather solution. This completes the proof of this theorem.
\end{proof}

\begin{remark}\label{re4.1}
\rm{Let $\mu(x)=2|x|^2+1, \rho(x)=\frac{2|x|^2+1}{(1-|x|^4e^{-|x|^4})^2}, v_q(x)=|x|^2e^{|x|^2}+1$. By a simple computation, it is easy to check that these functions satisfy the assumptions $(H1)-(H4)$. In this case, we can prove that the $2\pi$-periodic $\R^3$-valued breather solution exists in Theorem \ref{th4.1} which are exhibited in Figure \ref{f2}.}
\begin{figure}[h]
  \centering
  \includegraphics[width=6cm]{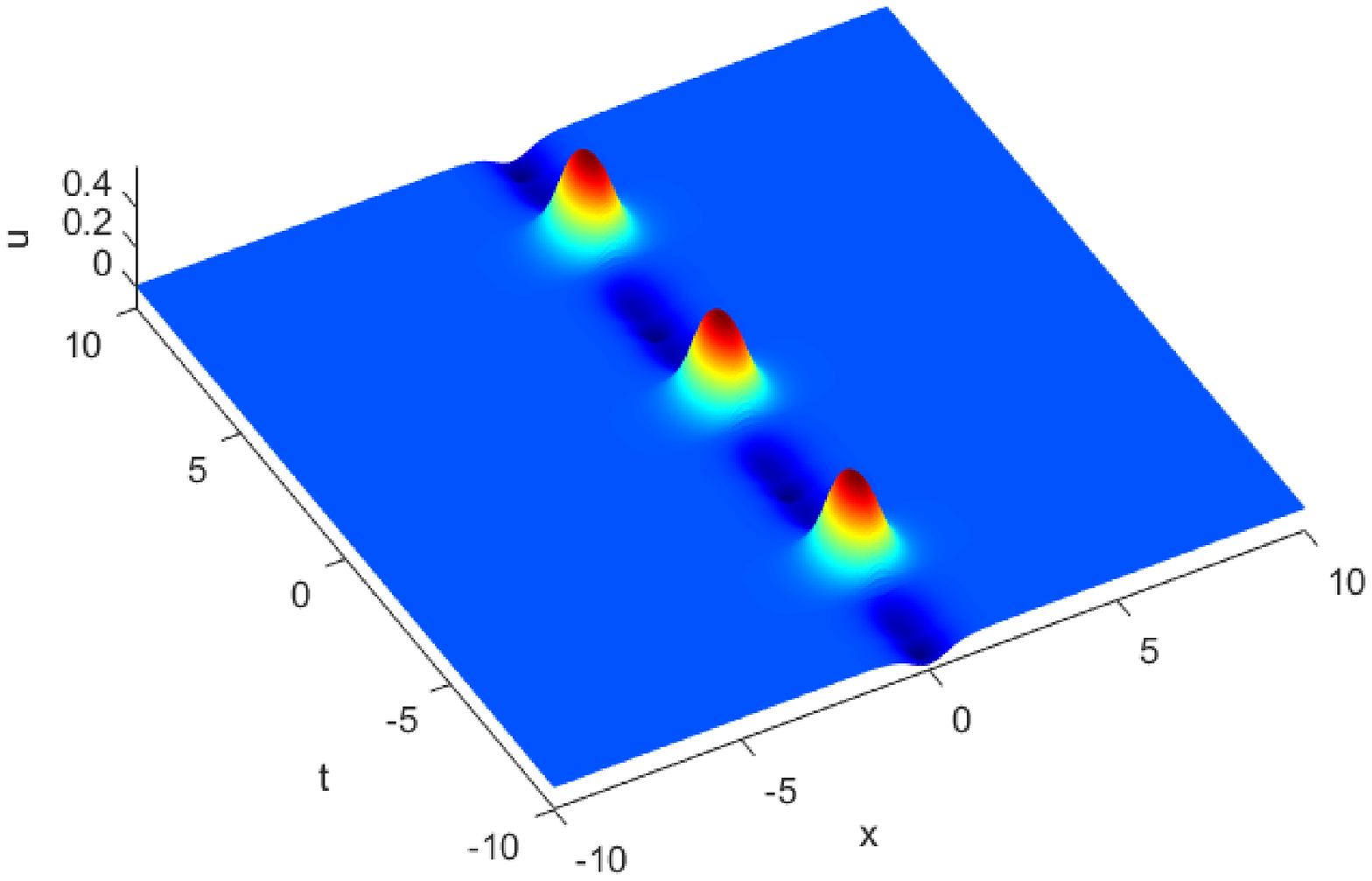}(a)\includegraphics[width=4.5cm]{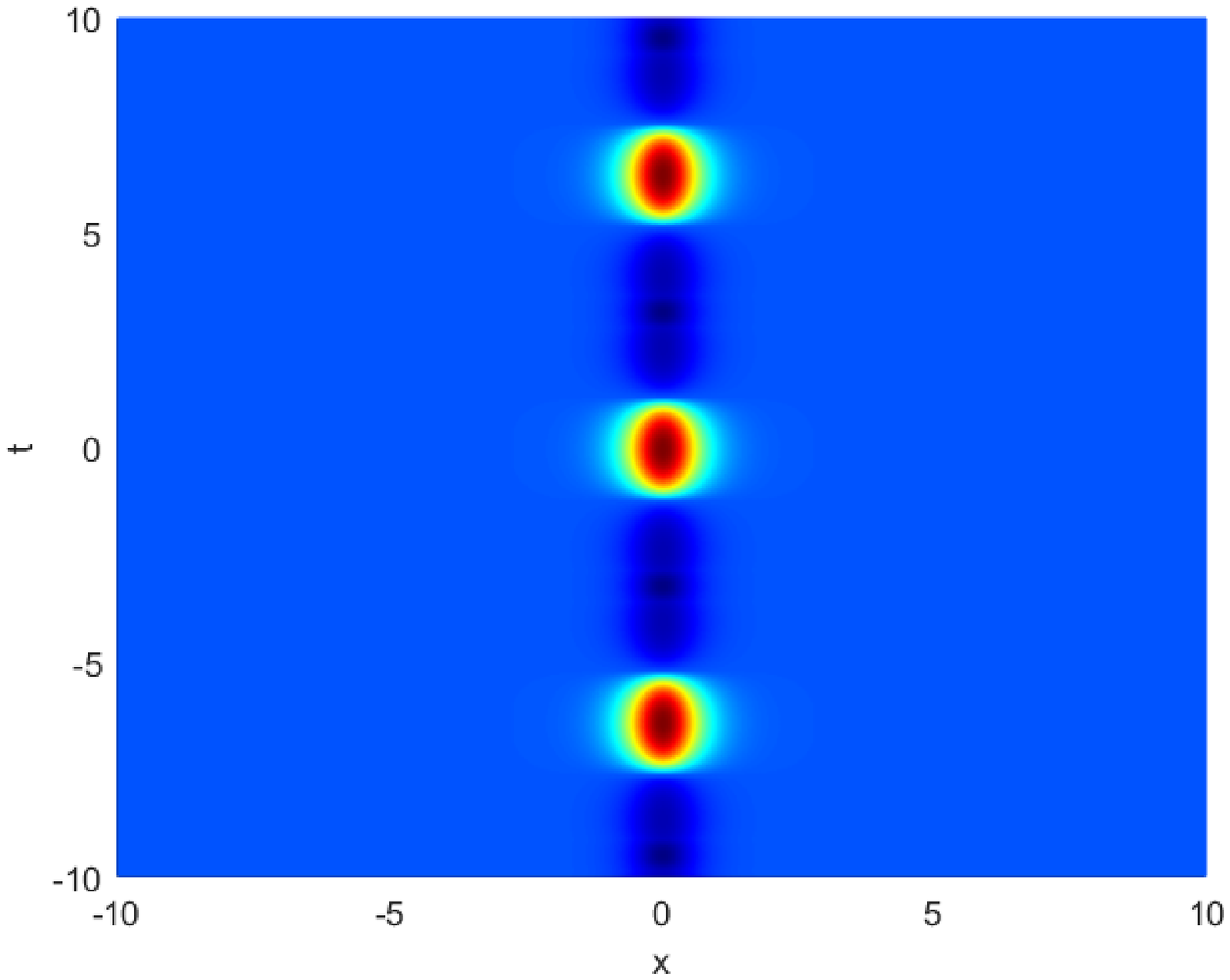}(b)\\
  \caption{(a) and (b) show part of the breather solution for p=3 and q=4.}\label{f2}
\end{figure}
\end{remark}

\begin{theorem}
Assume $\rho(x), \mu(x), v_p(x), v_q(x):\R^3\rightarrow(0,\infty)$ are radially symmetric $C^2$-functions such that conditions $(V)$ and $(H1)-(H4)$ hold and $T=2\pi\sigma(0)$. Then there exists a continuum of $T$-periodic $\mathbb{C}^3$-valued monochromatic breather solutions $\mathbf{u}(x,t)=e^{i\frac{2\pi}{T}t}\mathbf{u}(x)$ of equation \eqref{eq1.1} with $\sup_{\R^3}|\mathbf{u}(x)|e^{\gamma|x|}<\infty$.
\end{theorem}

\begin{proof}
We make use of the ansatz $\mathbf{u}(x,t)=y(|x|)e^{i\frac{2\pi}{T}t}\frac{x}{|x|}$ and insert it into \eqref{eq1.1}. It shows that a $T$-periodic breather solution $\mathbf{u}(x,t)$ exists if $y:[0,\infty)\rightarrow\R$ is a $C^2$-solution of
\begin{equation}\label{eq4.6}
-\left(\frac{2\pi}{T}\right)^2\tilde{\rho}(r)+\tilde{\mu}(r)+\tilde{v}_p(r)|y(r)|^{p-1}+\tilde{v}_q(r)|y(r)|^{q-1}=0,
\end{equation}
satisfying $y(0)=y''(0)=0$ which is exponentially decays to 0 at $\infty$. Note that \eqref{eq4.6} can be rewritten as
\begin{equation*}
(\tilde{\tau}(r))^{q-p}|y(r)|^{p-1}+|y(r)|^{q-1}=\left((\frac{2\pi}{T})^2 \frac{1}{(\tilde{\sigma}(r))^2}-1\right)(\tilde{\tau}(r))^{q-1},
\end{equation*}
where we have used \eqref{eq1.5}. We define $l(z):[0,\infty)\rightarrow\R$ as
\begin{equation*}
l(z)=(\tilde{\tau}(r))^{q-p}z^{p-1}+z^{q-1}.
\end{equation*}
Then
\begin{equation*}
l(|y(r)|)=(\tilde{\tau}(r))^{q-p}|y(r)|^{p-1}+|y(r)|^{q-1}
\end{equation*}
and
\begin{equation*}
l'(|y(r)|)=(p-1)(\tilde{\tau}(r))^{q-p}|y(r)|^{p-2}+(q-1)|y(r)|^{q-2}>0
\end{equation*}
for $1<p<q$. Thus we obtain that $l$ has an inverse $l^{-1}$, which implies that $y$ can be rewritten as
\begin{equation*}
y(r)=\pm l^{-1}\left[\left((\frac{2\pi}{T})^2\frac{1}{(\tilde{\sigma}(r))^2}-1\right)(\tilde{\tau}(r))^{q-1}\right].
\end{equation*}
The assumption $(H1)$ shows that $y$ is well-defined. In view of $(H3)$ and $(H4)$, $y(r)$ is exponentially decaying as $r\rightarrow\infty$ and by $(H2)$ we have $y(0)=y''(0)=0$. Thus we obtain that $\mathbf{u}(x,t)$ is a classical solution of \eqref{eq1.1} on $\R^3\times\R$. The proof of this theorem is completed.
\end{proof}

\section{Conclusion}
In this work, a (3+1)-dimensional radially symmetric curl-curl wave equation with double power nonlinearity is investigated. In view of the radially symmetric assumptions on the coefficients, the time periodic solutions and spatially localized real-valued solutions of equation \eqref{eq1.1} are constructed. Thus, the (3+1)-dimensional wave equation reduces to a family of ordinary differential equations. By means of qualitative theory of ODEs, we explored the existence and exponential decay properties of real-valued breather solutions of equation \eqref{eq1.1}. Furthermore, we showed that the solution can generate a continuum of phase-shifted breathers. In addition, we have also constructed $T$-periodic $\mathbb{C}^3$-valued monochromatic breather solutions of the type $\mathbf{u}(x,t)=e^{i\frac{2\pi}{T}t}\mathbf{u}(x)$ of equation \eqref{eq1.1}. We hope that our results can help enrich dynamic behaviors of the breathers under the context of the curl-curl nonlinear wave equations.

\begin{flushleft}
\textbf{Data Availability} This work does not have any experimental data.
\end{flushleft}

\begin{flushleft}
\textbf{Conflict of interest} We have no competing interests.
\end{flushleft}

\noindent \textbf{Acknowledgments} The authors sincerely thank the referees for very careful reading and many valuable comments, which led to much improvement in earlier version of this paper. This work is partially supported by NSFC Grants (12225103, 12071065 and 11871140) and the National Key Research and Development Program of China (2020YFA0713602 and 2020YFC1808301).\\

%

\bibliographystyle{elsarticle-num}
\bibliography{<your-bib-database>}

\end{document}